\documentclass[10pt,a4paper]{article}
\usepackage{latexsym,amsthm,amssymb,amscd,amsmath}
\newcounter{alphthm}
\usepackage[all]{xy}
\usepackage{stmaryrd}
\DeclareMathOperator{\llb}{\llbracket}
\DeclareMathOperator{\rrb}{\rrbracket}
\DeclareMathAlphabet{\mathpzc}{OT1}{pzc}{m}{it}

\newcommand{\im}{\mathrm Im}
\theoremstyle{plain}
\newtheorem{theorem}{Theorem}[section]
\newtheorem{lemma}[theorem]{Lemma}
\newtheorem{proposition}[theorem]{Proposition}
\newtheorem{cor}[theorem]{Corollary}
\theoremstyle{definition}

\newcommand{\be}{\begin{equation}}
\newcommand{\ee}{\end{equation}}
\newcommand{\ben}{\begin{enumerate}}
\newcommand{\een}{\end{enumerate}}

\newcommand{\wt}{\widetilde}
\title{Curvature collineations on Lie algebroid structure}
\author{C. M. Arcu\c{s}, E. Peyghan and E. Sharahi\\
Department of Mathematics, Faculty  of Science, Arak University,\\
 Arak 38156-8-8349,  Iran\\
Constantin M Arcu\c{s}\\
Secondary School "Cornelius Radu"\\
Radinesti Village, 217196\\
Gorj County, Romania\\
Emails: c\_arcus@radinesti.ro,\ e-peyghan@araku.ac.ir,\\
 esasharahi@gmail.com
}
\begin{document}
\maketitle
\begin{abstract}
Considering prolongation of a Lie algebroid equipped with a spray, defining some classical tensors, we show that a Lie symmetry of a spray is a curvature collineation for these tensors.
\end{abstract}
\section{Introduction}
The notion of a Lie algebroid structure as generalization of the notion of Lie algebra of a Lie group, introduced by Pradines
in 1967. The method of extracting a Lie algebroid from a differentiable groupoid, is completely analogous of extracting a Lie algebra from a Lie group\cite{pradines}. A Lie algebroid is a vector bundle that each of its sections is mapped to a vector field by a linear bundle map together with a bracket on the sections of the vector bundle that is $\mathbb{R}$-bilinear, alternating and satisfies the Jacobi identity. This map must be a homomorphism of Lie algebras and is called the anchor map of the vector bundle. More attributes on the anchor map, may induce special properties on the vector bundle. For example, if the anchor map is a submersion, then all of its right inverses are connections in the vector bundle (see \cite{makenzie}). When anchor map is the identity, the Lie algebroid reduces to the tangent bundle.  Thus Lie algebroids are extensions of the tangent bundle that make possible to study more generic geometric objects. There are many studies on Lie algebroid structures  (e.g., \cite{courant, grabowski1, grabowski2, makenzie}) and their relation to physics and mechanics (e.g. \cite{martinez, vacaru1, vacaru2, vacaru3}).

  Curvature collineations as symmetries of space-time, are powerful tools in general relativity \cite{aichelburg, hall}. It has been shown in \cite{szilasi1} that if the complete lift of a vector field is a Lie symmetry of a spray, then it is a curvature collineation for some classical tensors. The aim of this paper is to obtain the similar results on the prolongation of a Lie algebroid. In section 2 we recall the definition of prolongation of a Lie algebroid and review some basic concepts such as the vertical and complete lifts on the prolongation of a Lie algebroid. In section 3 we introduce a kind of derivation along projectable sections and in the last section we show that in a prolongation of a Lie algebroid the complete lift of a Lie symmetry of a spray, is a curvature collineation for some classical tensor fields.

 \section{Preliminaries}
 Let $\pi:E\rightarrow M$ be a vector bundle of rank $m$ over an $n$-dimensional base manifold $M$. Denote by $\Gamma(E)$ the $C^\infty(M)$-module of smooth sections of $\pi$.
A {\it{Lie algebroid structure}} $([., .]_E, \rho)$ on $E$ is a Lie bracket $[., .]_E$ on the module $\Gamma(E)$ together with a bundle map $\rho: E\rightarrow TM$, called the {\it{anchor map}}, such that we also denote by $\rho:\Gamma(E)\rightarrow\chi(M)$
the homomorphism of $C^\infty(M)$-modules induced by the anchor map and for being algebroid, the following role must be hold.
\[
[\xi, f\eta]_E=f[\xi, \eta]_E+\rho(\xi)(f)\eta,
\]
for all $\xi,\eta\in\Gamma(E), f\in C^\infty(M)$. Here we regard the anchor map as a homomorphism $\Gamma(E)\rightarrow\chi(M)$ of $C^\infty(M)$-modules, denoted by the same symbol. Then we also have
$$[\rho(\xi),\rho(\eta)]=\rho[\xi,\eta]_E,$$
so the anchor map $\rho:\Gamma(E)\rightarrow\chi(M)$ is a Lie algebra homomorphism at the same time \cite{martinez}.

On Lie algebroids $(E, [., .]_E, \rho)$  the differential of $E$, $d^E:\Gamma(\wedge^kE^*)\rightarrow\Gamma(\wedge^{k+1}E^*)$, is defined by
\begin{align*}
d^E\theta(\xi_0,\ldots, \xi_k)&=\sum_{i=0}^k(-1)^i\rho(\xi_i)(\mu(\xi_0,\ldots, \iota{\xi_i}, \ldots, \xi_k))\\
&\ \ \ +\sum_{i<j}(-1)^{i+j}\theta([\xi_i, \xi_j]_E, X_0, \ldots, \iota{\xi_i}, \ldots, \iota{\xi_j}, \ldots, \xi_k),
\end{align*}
for $\theta\in\Gamma(\wedge^kE^*)$ and $\xi_0, \ldots, \xi_k\in\Gamma(E)$, where the $\iota$ side an argument means the absence of that argument. In particular, if $f\in\Gamma(\wedge^0 E^*)=C^\infty(M)$ we have $d^Ef(\xi)=\rho(\xi)f$. By using the above equation one can deduce $(d^E)^2=0$. Moreover, for $\xi\in\Gamma(E)$, the contraction $i_\xi:\Gamma(\wedge^pE^*)\rightarrow\Gamma(\wedge^pE^*)$ is defined in the standard way and the Lie differential operator $\pounds_\xi^E:\Gamma(\wedge^pE^*)\rightarrow\Gamma(\wedge^pE^*)$ is defined by $\pounds_\xi^E=i_\xi\circ d^E+d^E\circ i_\xi$ \cite{grabowski1}.

If we take a local coordinate system $(x^i)_{i=1}^{n}$ on $M$ and a local basis $(e_\alpha)_{\alpha=1}^{m}$ of sections of $E$, then we have the corresponding local coordinate $(\textbf{x}^i, \textbf{y}^\alpha)$ on $E$, where $\textbf{x}^i:=x^i\circ\pi$ and $\textbf{y}^\alpha(u)$ is the $\alpha$-th coordinate of $u\in E$ in the given basis. Such coordinates determine local functions $\rho^i_\alpha$, $L^\gamma_{\alpha\beta}$ on $M$ which contain the local information of the Lie algebroid structure, and accordingly they are called the structure functions of the Lie algebroid \cite{martinez}. These functions are given by
\[
\rho(e_\alpha)=\rho^i_\alpha\frac{\partial}{\partial x^i}\qquad\text{and}\ \ \ [e_\alpha, e_\beta]_E=L^\gamma_{\alpha\beta}e_\gamma.
\]
An easy calculation leads to the structure equations
\begin{equation}\label{2}
(i)\ \rho^j_\alpha\frac{\partial\rho^i_\beta}{\partial x^j}-\rho^j_\beta\frac{\partial\rho^i_\alpha}{\partial x^j}=\rho^i_\gamma L_{\alpha\beta}^\gamma,\ \ \ (ii)\ \sum_{(\alpha, \beta, \gamma)}[\rho^i_\alpha\frac{\partial L^\nu_{\beta\gamma}}{\partial x^i}+L^\nu_{\alpha\mu}L^\mu_{\beta\gamma}]=0,
\end{equation}
The {\it{vertical}} lift of a function $f\in C^{\infty}(M)$ is $f^\vee:=f\circ\pi\in C^{\infty}(E)$. The vertical lift $\xi^\vee$ of a section $\xi\in\Gamma(E)$ is given by
$$u\in E\longmapsto \xi^\vee(u):=\xi(\pi(u))^{\vee}_{u}\in TE,$$
where ${}^{\vee}_u: E_{\pi(u)}\rightarrow T_u (E_{\pi(u)})$ is the canonical isomorphism between the vector spaces $E_{\pi(u)}$ and $T_uE_{\pi(u)}$(see, e.g., \cite{szilasi2}, 2.4.(5)). Then $\xi^\vee$ is a vertical vector field on $E$. Thus it follows that if $\xi=\xi^\alpha e_\alpha\in\Gamma(E)$, then the vertical lift $\xi^\vee$ has the locally expression $\xi^\vee=(\xi^\alpha\circ\pi)\frac{\partial}{\partial\textbf{y}^\alpha}$. If $\xi$, $\eta$ are sections of $E$ and $f\in C^\infty(M)$, then using the local expressions of them, we obtain \cite{peyghan}
\[
(\xi+\eta)^\vee=\xi^\vee+\eta^\vee,\ \ \ (f\xi)^\vee=f^\vee \xi^\vee,\ \ \ \xi^\vee f^\vee=0.
\]
The {\it{complete lift}} of a smooth function $f\in C^{\infty}(M)$ into $C^{\infty}(E)$ is the smooth function
\[
f^c:E\longrightarrow\mathbb{R},\ \ \ \ u\longmapsto f^c(u):=\rho(u)f.
\]
Then
\[
(fg)^c=f^cg^\vee+f^\vee g^c,
\]
because for every $u\in E$
\begin{align*}
(fg)^c(u)&=\rho(u)(fg)=(\rho(u)f)(g\circ\pi)(u)+(f\circ\pi)(u)(\rho(u)g)\\
&=f^c(u)g^\vee(u)+f^\vee(u)g^c(u).
\end{align*}
Locally we have
\begin{align*}
f^c(u)=f^c(u^\alpha e_\alpha)=\rho(u^\alpha e_\alpha)(f)=u^\alpha\rho(e_\alpha)(f)=u^\alpha\rho^i_\alpha\frac{\partial f}{\partial x^i}=(\textbf{y}^\alpha((\rho^i_\alpha\frac{\partial f}{\partial x^i})\circ\pi))(u),
\end{align*}
i.e., $f^c|_{\pi^{-1}(U)}=\textbf{y}^\alpha((\rho^i_\alpha\frac{\partial f}{\partial x^i})\circ\pi)$.
\begin{lemma}\cite{peyghan}
If $\xi$ is a section of $E$ and $f, g\in C^\infty(M)$, then
\[
\ (f+g)^c=f^c+g^c,\ \ \ \ (fg)^c=f^cg^\vee+f^\vee g^c,\ \ \ \ \xi^\vee f^c=(\rho(\xi)f)^\vee.
\]
\end{lemma}
We refer that every smooth section $\omega$ of the dual bundle of $\pi:E\longrightarrow M$ determines a smooth function $\widehat{\omega}\longrightarrow M$ given by
$$\widehat{\omega}(u):=\omega_{\pi(u)}(u).$$
Now let $\xi$ be a smooth section of $E$. There exist a unique vector field $\xi^c$, called the {\it{complete lift}} of $\xi$, such that

\textbf{i}) $\xi^c$ is $\pi$-projectable on $\rho(\xi)$,

\textbf{ii}) $\xi^c(\hat{\theta})=\widehat{\pounds_{\xi}^{E}\theta}$,\\
where $\theta \in \Gamma(E^*)$.

It is known that $\xi^c$ has the following coordinate expression(\cite{grabowski1}, \cite{grabowski2}):
\begin{equation}\label{esa}
\xi^c=\{(\xi^\alpha \rho_{\alpha}^{i})\circ \pi\} \frac{\partial}{\partial \textbf{x}^i}+\textbf{y}^\beta \{(\rho_{\beta}^{j}\frac{\partial \xi^\alpha}{\partial x^j}-\eta^\gamma L^\alpha_{\gamma\beta})\circ\pi\} \frac{\partial}{\partial \textbf{y}^\alpha}.
\end{equation}
\begin{lemma}
\cite{peyghan} If $\xi$ and $\eta$ are sections of $E$ and $f\in C^\infty(M)$, then
\begin{itemize}
\item[(i)] $\xi^cf^c=(\rho(\xi)f)^c,\ \ \ \text{for all}\ \ \ \ f\in C^\infty(M),$
\item[(ii)] $\xi^cf^\vee=(\rho(\xi)f)^\vee,$
\item[(iii)] $[\xi^c, \eta^c]=[\xi, \eta]^c_E,\ \ \ [\xi^c, \eta^\vee]=[\xi, \eta]^{\vee}_E,\ \ \ [\xi^\vee, \eta^\vee]=0.$
\end{itemize}
\end{lemma}
\subsection{The prolongation of a Lie algebroid}
In this section we recall the notion of {\it{prolongation}} of a Lie algebroid and we consider a Lie algebroid structure on it. We also study the vertical and complete lifts on the prolongation of a Lie algebroid.

Let $\pounds^\pi E$ be the subset of $E\times TE$ defined by $\pounds^\pi E=\{(u, z)\in E\times TE|\rho(u)=\pi_*(z)\}$ and
let $\pi_\pounds: \pounds^\pi E\rightarrow E$ be the mapping given by $\pi_\pounds(u, z)=\tau_E(z)$, where $\tau_E: TE\rightarrow E$ is the natural projection. Then $(\pounds^\pi E, \pi_\pounds, E)$ is a vector bundle over $E$ of rank $2n$. If we define an anchor map $\rho_\pounds:\pounds^\pi E\rightarrow TE$ on $\pounds^\pi E$, then  this vector bundle becomes a Lie algebroid with structure $(\llb., .\rrb, \rho_\pounds)$, where the Lie bracket $\llb .,.\rrb$ is given by formula (17) in \cite{peyghan}.

We introduce the vertical subbundle
\[
v\pounds^\pi E=\ker\tau_\pounds=\{(u, z)\in \pounds^\pi E|\tau_\pounds(u, z)=0\},
\]
of $\pounds^\pi E$ where $\tau_\pounds:\pounds^\pi E\rightarrow E$ is the projection onto the first factor, i.e., $\tau_\pounds(u, z)=u$. Then the elements of $v\pounds^\pi E$ are of the form $(0,z)\in E\times TE$ such that $\pi_*(z)=0$, these elements are called {\it{vertical}}. Since $\pi_*(z)=0$ and $\ker\pi_*=vE$ ($\pi_*:TE\rightarrow TM$), then we deduce that if $(0, z)$ is vertical then $z$ is a vertical vector on $E$ \cite{peyghan}.

If we consider a local base $\{e_\alpha\}$ of sections of $E$ and coordinates $(\textbf{x}^i, \textbf{y}^\alpha)$ on $E$, then we have local  coordinates $(\textbf{x}^i, \textbf{y}^\alpha, k^\alpha, z^\alpha)$ on $\pounds^\pi E$ given as follows. If $(u, z)$ is an element of $\pounds^\pi E$, then by using $\rho(u)=\pi_*(z)$, $z$ has the form
\[
z=((\rho^i_\alpha u^\alpha)\circ\pi){\frac{\partial}{\partial\textbf{x}^i}}|_v+z^\alpha{\frac{\partial}{\partial\textbf{y}^\alpha}}|_v,\ \ \ z\in T_vE.
\]
The local base $\{\mathcal{X}_\alpha, \mathcal{V}_\alpha\}$ of sections of $\pounds^\pi E$ associated to the coordinate system is given by
\[
\mathcal{X}_\alpha(v)=(e_\alpha(\pi(v)), (\rho^i_\alpha\circ\pi)\frac{\partial}{\partial\textbf{x}^i}|_v),\ \ \ \mathcal{V}_\alpha(v)=(0, \frac{\partial}{\partial\textbf{y}^\alpha}|v).
\]
If $\wt{\eta}$ is a section of $\pounds^\pi E$ by coordinate expression
\[
\wt{\eta}(x, y)=(x^i, y^\alpha, Z^\alpha(x, y), V^\alpha(x, y)),
\]
then the expression of $\wt{\eta}$ in terms of base $\{\mathcal{X}_\alpha, \mathcal{V}_\alpha\}$ is 
\[
\wt{\eta}=Z^\alpha\mathcal{X}_\alpha+V^\alpha\mathcal{V}_\alpha.
\]
\begin{lemma}\cite{peyghan} The followings are hold.
\[
\llb\mathcal{X}_\alpha, \mathcal{X}_\beta\rrb=(L^\gamma_{\alpha\beta}\circ\pi)\mathcal{X}_\gamma,\ \ \ \llb\mathcal{X}_\alpha, \mathcal{V}_\beta\rrb=0,\ \ \ \llb\mathcal{V}_\alpha, \mathcal{V}_\beta\rrb=0.
\]
\end{lemma}
\subsubsection{Vertical and complete lifts on $\pounds^\pi E$}
The vertical lift $\eta^V$ and the complete lift $\eta^C$ of a section $\eta\in\Gamma(E)$ as the sections of $\pounds^\pi E\rightarrow E$ are defined by
\[
\eta^V(u)=(0, \eta^v(u)),\ \ \ \eta^C(u)=(\eta(\pi(u)), \eta^c(u)),\ \ \ u\in E.
\]
It is shown that vertical and complete lifts has the coordinate expression
\begin{equation}\label{esi}
\eta^V=(\eta^\alpha\circ\pi)\mathcal{V}_\alpha,\ \ \ \eta^C=(\eta^\alpha\circ\pi)\mathcal{X}_\alpha+\textbf{y}^\beta[(\rho^j_\beta\frac{\partial \eta^\alpha}{\partial x^j}-\eta^\gamma L^\alpha_{\gamma\beta})\circ\pi]\mathcal{V}_\alpha,
\end{equation}
where $\eta=\eta^\alpha e_\alpha\in\Gamma(E)$ \cite{peyghan}. Note that the first components of $\eta^V$ and $\eta^C$ are free of indeterminacy. Setting $\eta=e_\alpha$, we have
\begin{equation}
\eta^C=\mathcal{X}_\alpha-y^\beta(L^{\gamma}_{\alpha\beta}\circ\pi)\mathcal{V}_\gamma.
\end{equation}
Here we consider the pullback bundle $(\pi^*E, pr_1, E)$ of the vector bundle $(E,\pi,M)$, where
$$\pi^*E:=E\times_ME=\lbrace (u,v)\in E\times E| \pi(u)=\pi(v)\rbrace,$$
and $pr_1$ is the projection map onto the first component. We also consider the sequence
\begin{equation}\label{exact se}
0\longrightarrow E\times_ME\stackrel{i}{\rightarrow}\pounds^\pi E\stackrel{j}{\rightarrow}E\times_ME\longrightarrow0,
\end{equation}
with $j(u, z)=(\pi_E(z), Id(u))=(v, u)$, $z\in T_vE$, and $i(u, v)=(0, v^\vee_u)$ where $v^\vee_u:C^\infty(E)\rightarrow\mathbb{R}$ is defined by $v^\vee_u(F)=\frac{d}{dt}|_{t=0}F(u+tv)$. Indeed we have $v^\vee_u=\frac{d}{dt}|_{t=0}(u+tv)$. Function $J=i\circ j:\pounds^\pi E\rightarrow\pounds^\pi E$ is called the {\it {vertical
endomorphism}} ({\it{almost tangent structure}}) of $\pounds^\pi E$. For any section $\eta$ on $E$, the map
\[
\widehat{\eta}:E\rightarrow\pi^*E,
\]
defined by $\widehat{\eta}(u)=(u, \eta\circ\pi(u))$ is a section of $\pi^*\pi$, called the lift of $\eta$ into $\Gamma(\pi^*\pi)$. $\widehat{\eta}$ may be identified with the map $\eta\circ\pi:E\rightarrow E$. It is easy to see that $\{\widehat{\eta}|\eta\in\Gamma(E)\}$ generates locally the $C^\infty(E)$-module $\Gamma(\pi^*\pi)$. It is obvious that $i(\widehat{\eta})=\eta^V$, $j(\eta^V)=0$ and $j(\eta^C)=\widehat{\eta}$. Moreover, $i$ is injective and $j$ is surjective. Therefore the sequence given by (\ref{exact se}) is an exact sequence. Moreover, if $\{\mathcal{X}^\alpha, \mathcal{V}^\alpha\}$ be the corresponding dual basis of $\{\mathcal{X}_\alpha, \mathcal{V}_\alpha\}$, then $J=\mathcal{V}_\alpha\otimes\mathcal{X}^\alpha$ (see \cite{peyghan}). One can derive from the above exact sequence that
\[
\im J=\im i=v\pounds^\pi E,\ \ \ \ker J=\ker j=v\pounds^\pi E,\ \ \ J\circ J=0.
\]
The section $C : E \longrightarrow \pounds^\pi E$ given by $C:=i\circ\delta$, is called {\it Liouville} or {\it Euler section}, where $\delta:u\in E\rightarrow\delta(u)=(u, u)\in E\times_ME$. The Liouville section $C$ has the coordinate expression $C=\textbf{y}^\alpha\mathcal{V}_\alpha$, with respect to $\{\mathcal{X}_\alpha, \mathcal{V}_\alpha\}$. Section $\wt{\eta}$ of  vector bundle $(\pounds^\pi E, \pi_\pounds, E)$ is said to be {\it homogenous of degree} $r$, where $r$ is an integer, if $\llb C, \wt{\eta}\rrb=(r-1)\wt{\eta}$. A function $h:\pounds^\pi E\rightarrow \pounds^\pi E$ is called a {\it horizontal endomorphism} if $h\circ h=h$ and $\ker h=v\pounds^\pi E$. Also, $v:=Id-h$ is called {\it vertical projector associated to} $h$. Setting $h\pounds^\pi E:=\im h$  and using the fact that $\pounds^\pi E=\ker h+\im h=v\pounds^\pi E\oplus h\pounds^\pi E$, it will be deduced that
\begin{equation}\label{good}
\pounds^\pi E=v\pounds^\pi E\oplus h\pounds^\pi E.
\end{equation}
Thus one can check the following equations.
\begin{equation}\label{Jh}
(i)\ hJ=hv=Jv=0,\ \ (ii)\ v\circ v=v,\ \ (iii)\ vh=0,\ \ (iv)\ Jh=J=vJ.
\end{equation}
Moreover, $h$ has the locally expression $h=(\mathcal{X}_\beta+\mathcal{B}^\alpha_\beta\mathcal{V}_\alpha)\otimes\mathcal{X}^\beta$ (see \cite{peyghan}). Let $\eta$ be a section on $E$. The horizontal lift of $\eta$ by $h$ is a section of $\pounds^\pi E$  defined by $\eta^h=h(\eta^C)$. If we set $\delta_\alpha=e_\alpha^h$, then we have $\delta_\alpha=\mathcal{X}_\alpha+\mathcal{B}^\beta_\alpha\mathcal{V}_\beta=h(\mathcal{X}_\alpha)$. It is easy to see that $h\delta_\alpha=\delta_\alpha$, $v\delta_\alpha=0$ and
\begin{equation}\label{rho}
\rho_\pounds(\delta_\alpha)=(\rho^i_\alpha\circ\pi)\frac{\partial }{\partial \textbf{x}^i}+\mathcal{B}^\gamma_\alpha\frac{\partial}{\partial \textbf{y}^\gamma}.
\end{equation}
Moreover, $\{\delta_\alpha\}$ generate a basis of $h\pounds^\pi E$ and the frame $\{\delta_\alpha, \mathcal{V}_\alpha\}$ is a local basis of $\pounds^\pi E$ adapted to splitting (\ref{good}) which is called {\it  adapted basis}. The dual adapted basis is $\{\mathcal{X}^\alpha, \delta\mathcal{V}^\alpha\}$, where
\[
\delta\mathcal{V}^\alpha=\mathcal{V}^\alpha-\mathcal{B}^\alpha_\beta\mathcal{X}^\beta.
\]
\begin{lemma}
\cite{peyghan} The Lie brackets of the adapted basis $\{\delta_\alpha, \mathcal{V}_\alpha\}$  are
\begin{equation}\label{LieB}
\llb\delta_\alpha, \delta_\beta\rrb=(L^\gamma_{\alpha\beta}\circ\pi)\delta_\gamma+R^\gamma_{\alpha\beta}\mathcal{V}_\gamma,\ \ \ \llb\delta_\alpha, \mathcal{V}_\beta\rrb=-\frac{\partial \mathcal{B}^\gamma_\alpha}{\partial\textbf{y}^\beta}\mathcal{V}_\gamma,\ \ \ \llb\mathcal{V}_\alpha, \mathcal{V}_\beta\rrb=0,
\end{equation}
where
\begin{equation}\label{curv0}
R^\gamma_{\alpha\beta}=(\rho_\alpha^i\circ\pi)\frac{\partial{\mathcal{B}^\gamma_\beta}}{\partial{\textbf{x}^i}}
-(\rho_\beta^i\circ\pi)\frac{\partial{\mathcal{B}^\gamma_\alpha}}
{\partial{\textbf{x}^i}}+\mathcal{B}^\lambda_{\alpha}\frac{\partial{B^\gamma_\beta}}
{\partial{\textbf{y}^\lambda}}
-\mathcal{B}^\lambda_{\beta}\frac{\partial{\mathcal{B}^\gamma_\alpha}}{\partial{\textbf{y}^\lambda}}
+(L^\lambda_{\beta\alpha}\circ\pi)\mathcal{B}^\gamma_\lambda.
\end{equation}
\end{lemma}
Thus, one can immediately check that $h$ has the following coordinate expression with respect to the adapted basis
\begin{equation}\label{hF}
 h=\delta_\alpha\otimes\mathcal{X}^\alpha.
\end{equation}
\begin{lemma}
\cite{peyghan} The followings are hold.
\[
v\mathcal{V}_\alpha=\mathcal{V}_\alpha,\ \ \ v\mathcal{X}_\alpha=-\mathcal{B}^\beta_\alpha\mathcal{V}_\beta.
\]
\end{lemma}
\section{Derivative along projectable sections}
In this section considering projectable sections, a derivative along them is introduced.

A section $S$ of the vector bundle $(\pounds^\pi E, \pi_\pounds, E)$ is said to be a {\it semispray} if it satisfies
the condition $J(S)=C$. It is easy to see that local form of a semispray $S$ is
\begin{equation}\label{semispray}
S={\textbf{y}}^\alpha\mathcal{X}_\alpha+S^\alpha\mathcal{V}_\alpha.
\end{equation}
Moreover if $S$ is homogeneous of degree 2, i.e., $\llb C,S\rrb=S$, then we call it {\it spray}. It is also easy to see that $S$ is spray if and only if $2S^\beta=y^\alpha\dfrac{\partial S^\beta}{\partial \textbf{y}^\alpha}$.

We can deduce the following exact sequence from the exact sequence (\ref{exact se})
\begin{equation}\label{exact2}
0\longrightarrow \Gamma(\pi^*\pi)\stackrel{i}{\rightarrow}\Gamma(\pounds^\pi E)\stackrel{j}{\rightarrow}\Gamma(\pi^*\pi)\longrightarrow0.
\end{equation}
A right splitting of the exact sequence (\ref{exact2}), is an \textit{Ehresmann connection}. In other words any smooth $C^{\infty}(E)$-linear mapping $\mathcal{H}$ such that $j\circ \mathcal{H}=1_{\Gamma(\pounds^\pi E))}$, is an Ehresmann connection.

The vertical mapping $\mathcal{V}$ associated to $\mathcal{H}$ is a left splitting of (\ref{exact2}) that satisfies in $ker(\mathcal{V})=Im(\mathcal{H})$. Moreover, we have $\mathcal{V}(\mathcal{V}_\alpha)=\widehat{e_\alpha}$ and $\mathcal{V}(\mathcal{X}_\alpha)=-\mathcal{B}^{\beta}_{\alpha}\widehat{e_\beta}$.

It is easy to see that $h=\mathcal{H}\circ j$ and $v=i\circ\mathcal{V}=1_{\Gamma(\pounds^\pi E)}-h$ are the horizontal endomorphism and the vertival projection on $\pounds^\pi E$, respectively. Moreover, we have $\mathcal{H}(\widehat{\eta})=h(\eta^C)=\eta^h$. The Ehresmann connection $\mathcal{H}$ is said to be \textit{homogeneous} if $\llb C,\eta^h\rrb=0$ for all $\eta \in \Gamma(E)$.

We can derive a homogeneous Ehresmann connection $\mathcal{H}$ in $E$ from a spray $S$ on $E$ such that for any $\eta\in \Gamma(E)$
$$\eta^h=\mathcal{H}\widehat{\eta}=\dfrac{1}{2}(\eta^C+\llb \eta^V,S\rrb).$$
This Ehresmann connection is called the \textit{Berwald connection}. In this manner, it will bee seen the following (see \cite{peyghan})
\begin{equation}\label{mokh}
2\mathcal{B}^\gamma_\alpha=\dfrac{\partial S^\gamma}{\partial {\textbf{y}}^\alpha}-y^\beta(L^\gamma_{\alpha\beta}\circ\pi).
\end{equation}

We define the \textit{vertical differential} of a function $F \in C^{\infty}(E)$ and a section $\bar{\eta}\in E\times_ME$ as the $1$-form $\nabla^{v}F$ and $(1,1)$-tensor $\nabla^{v}\bar{\eta}$ as follows
$$\nabla^{v}F\bar{\xi}:=\rho_\pounds(i\bar{\xi})F,\qquad \bar{\xi}\in E\times_ME,$$
and
\begin{equation}\label{aa}
\nabla^{v}\bar{\eta}(\bar{\xi}):=\nabla^{v}_{\bar{\xi}}\bar{\eta}:=j\llb i\bar{\xi},\wt{\sigma}\rrb; \qquad \bar{\xi}\in E\times_ME, \quad \wt{\sigma} \in \pounds^\pi E, \quad j\wt{\sigma}=\bar{\eta}.
\end{equation}
It is easy to see that
$$\nabla^{v}F\widehat{e_\alpha}=\dfrac{\partial F}{\partial \textbf{y}^\alpha}.$$
For an $\eta\in\Gamma(E)$, we define a $(1,1)$-tensor field by
\begin{equation}
 [J,\eta]^{F-N}\xi:=\llb J\xi,\eta\rrb-J\llb\xi,\eta\rrb,  \qquad \xi\in\Gamma(E).
\end{equation}
In particular with acting on $\eta^C$ and $\eta^V$ we have the following
\begin{equation}
 [J,\eta^C]^{F-N}= [J,\eta^V]^{F-N}=0.
\end{equation}
In terms of Ehresmann connection, we can rewrite
\begin{equation}
\nabla^{v}_{\bar{\xi}}\bar{\eta}=j\llb i\bar{\xi},\mathcal{H}\bar{\eta}\rrb.
\end{equation}
Setting $\bar{\xi}=\bar{\xi}^{\alpha}\widehat{e_{\alpha}}$ and $\bar{\eta}=\bar{\eta}^{\beta}\widehat{e_{\beta}}$, one can see that
\begin{equation}\label{l1}
\nabla^{v}_{\bar{\xi}}\bar{\eta}=\bar{\xi}^{\alpha}\dfrac{\partial \bar{\eta}^{\beta}}{\partial {\textbf{y}}^{\alpha}}j(\delta_{\beta})=\bar{\xi}^{\alpha}\dfrac{\partial \bar{\eta}^{\beta}}{\partial {\textbf{y}}^{\alpha}}\widehat{e}_\beta.
\end{equation}
Using this equation, one can deduce the following
\begin{equation}\label{20}
\nabla^{v}_{\bar{\xi}}\widehat{\eta}=0.
\end{equation}
For any Ehresmann connection $\mathcal{H}$, similar to the vertical defferential we can define the \textit{$h$-Berwald defferential} $\nabla^h$ as follows
\begin{equation}
\nabla^hF(\bar{\xi}):=\rho_\pounds(\mathcal{H}\bar{\xi})F,
\end{equation}
and
\begin{equation}
\nabla^h\bar{\eta}(\bar{\xi}):=\nabla^{h}_{\bar{\xi}}\bar{\eta}:=\mathcal{V}\llb\mathcal{H}\bar{\xi}, i\bar{\eta} \rrb.
\end{equation}
Using expression of $\rho(\delta_\alpha)$, we can compute
\begin{equation}
\nabla^hF(\widehat{e_\alpha})=\rho_{\alpha}^{i}\dfrac{\partial F}{\partial \textbf{x}^i}+\mathcal{B}_{\alpha}^{\gamma}\dfrac{\partial F}{\partial {\textbf{y}}^\gamma}.
\end{equation}
Moreover setting $\bar{\xi}=\bar{\xi}^{\alpha}\widehat{e_\alpha}$ and $\bar{\eta}=\bar{\eta}^{\alpha}\widehat{e_\alpha}$, we have
\begin{align}\label{Gav1}
\nabla^{h}_{\bar{\xi}}\bar{\eta}&=\mathcal{V}\llb \mathcal{H}(\bar{\xi}^{\alpha}\widehat{e_\alpha}), i(\bar{\eta}^{\beta}\widehat{e_\beta})\rrb=\mathcal{V}\llb \bar{\xi}^\alpha\delta_\alpha,\bar{\eta}^\beta\mathcal{V}_\beta\rrb\nonumber\\
&=\mathcal{V}\lbrace \bar{\xi}^\alpha \bar{\eta}^\beta\llb\delta_\alpha, \mathcal{V}_\beta\rrb+\bar{\xi}^\alpha(\rho(\delta_\alpha)\bar{\eta}^\beta)\mathcal{V}_\beta-\bar{\eta}^\beta(\rho(\mathcal{V_\beta})\bar{\xi}^\alpha)\delta_\alpha\rbrace
\nonumber\\
&=-\bar{\xi}^\alpha \bar{\eta}^\beta\dfrac{\partial {\mathcal{B}}^{\gamma}_{\alpha}}{\partial {\textbf{y}}^\beta}\mathcal{V}_\gamma+\bar{\xi}^\alpha(\rho^{i}_{\alpha}\circ\pi)\dfrac{\partial \bar{\eta}^\beta}{\partial \textbf{x}^i}\mathcal{V}_\beta+\bar{\xi}^\alpha{\mathcal{B}}^{\gamma}_{\alpha}\dfrac{\partial \bar{\eta}^\beta}{\partial {\textbf{y}}^\gamma}\mathcal{V}_\beta\nonumber\\
&=\lbrace-\bar{\xi}^\alpha \bar{\eta}^\beta\dfrac{\partial {\mathcal{B}}^{\gamma}_{\alpha}}{\partial {\textbf{y}}^\beta}+\bar{\xi}^\alpha(\rho^{i}_{\alpha}\circ\pi)\dfrac{\partial \bar{\eta}^\gamma}{\partial \textbf{x}^i}+\bar{\xi}^\alpha{\mathcal{B}}^{\beta}_{\alpha}\dfrac{\partial \bar{\eta}^\gamma}{\partial {\textbf{y}}^\beta}\rbrace\mathcal{V}_\gamma.
\end{align}
An $\mathbb{R}$-linear mapping $D:\mathfrak{T}(\pi^*\pi)\longrightarrow\mathfrak{T}(\pi^*\pi)$ preserving type and commutes with contraction and satisfied the
$$D(A\otimes B)=(DA)\otimes B+A\otimes(DB); \qquad A,B\in \mathfrak{T}(\pi_\pounds),$$
where $\mathfrak{T}(\pi^*\pi)$ denotes the family of all tensor fields of the bundle $\pi^*\pi$, is called a \textit{derivation along $\pi_\pounds$}. If we have enough data on $C^{\infty}(E)$ and $\Gamma(E)$, then we can build a tensor derivation\cite{onil}.
\begin{lemma}\label{Willmor}
Any derivation of $\mathfrak{T}(\pi^*\pi)$ is completely determined by its action on $C^{\infty}(E)$ under the anchor map and its action on $\Gamma(\pi^*\pi)$. Converesly, given a section $\wt{\eta}\in \Gamma(\pounds^\pi E)$ and an $\mathbb{R}$-linear mapping $D_0: \Gamma(\pi^*\pi)\longrightarrow \Gamma(\pi^*\pi)$
 such that
$$D_0(F\bar{\xi})=(\rho_\pounds(\wt{\eta})(F))\bar{\xi}+FD_0\bar{\xi}, \qquad \bar{\xi}\in \Gamma(\pi^*\pi), \qquad F\in C^{\infty}(E),$$
there exist a unique derivation $D$ along $\pi_\pounds$ such that $D\upharpoonleft C^{\infty}(E)=\rho_\pounds(\wt{\eta})$ and $D\upharpoonleft \Gamma(\pi^*\pi)=D_0$.
\end{lemma}
A section $\wt{\xi}: E \longrightarrow \pounds^\pi(E)$ is said to be \textit{projectable} if there is a section $X: M\longrightarrow E$  such that $\tau_\pounds\circ \wt{\xi}=X\circ\pi$. It is easy to see that both $\xi^V$ and $\xi^C$ are projectable when $\xi\in\Gamma(E)$.
\begin{lemma}
If $\wt{\xi}: \pounds^\pi E\longrightarrow E$ be projectable, then there exist a unique derivation $\widetilde{\mathcal{L}}_{\wt{\xi}}$ along $\pi_\pounds$ such that
\begin{equation}\label{shahin}
\widetilde{\mathcal{L}}_{\wt{\xi}} F:=\rho_\pounds({\wt{\xi}}) F, \qquad F\in C^{\infty}(E),
\end{equation}
\begin{equation}\label{wd}
\widetilde{\mathcal{L}}_{\wt{\xi}}\bar{\eta}:=i^{-1}\llb {\wt{\xi}}, i\bar{\eta}\rrb, \qquad \bar{\eta}\in \Gamma(\pi^*\pi).
\end{equation}
\end{lemma}
\begin{proof}
Since ${\wt{\xi}}$ is projectable and $i\bar{\eta}$ is vertical, then $\llb{\wt{\xi}},i\bar{\eta}\rrb$ is vertical too. Therefore (\ref{wd}) is a well defined equivalency. Now if $F\in C^{\infty}(E)$, then
\begin{align*}
\widetilde{\mathcal{L}}_{\wt{\xi}} F\bar{\eta}&:=i^{-1}\llb {\wt{\xi}}, i(F\bar{\eta})\rrb=i^{-1}\llb {\wt{\xi}},Fi \bar{\eta}\rrb=i^{-1}(F\llb {\wt{\xi}}, i\bar{\eta}\rrb+\rho_\pounds({\wt{\xi}})(F)(i\bar{\eta}))\\
&=F\widetilde{\mathcal{L}}_{\wt{\xi}}\bar{\eta}+\rho_\pounds({\wt{\xi}})(F)\bar{\eta}.
\end{align*}
Now using lemma (\ref{Willmor}), we prove the assertion.
\end{proof}
We call the $\widetilde{\mathcal{L}}_{\wt{\xi}}$ as \textit{Lie derivation along $\pi$} with respect to ${\wt{\xi}}$. Note that if ${\wt{\xi}}$ be projectable, then $\llb {\wt{\xi}}, i \bar{\eta}\rrb=v\llb {\wt{\xi}}, i \bar{\eta}\rrb=i\circ \mathcal{V}\llb {\wt{\xi}}, i \bar{\eta}\rrb$. Thus we can write (\ref{wd}) as
\begin{equation}\label{eshgh}
\widetilde{\mathcal{L}}_{\wt{\xi}}\bar{\eta}=\mathcal{V}\llb {\wt{\xi}}, i \bar{\eta}\rrb.
\end{equation}
Setting $\wt{\xi}={{\wt{\xi}}}^\alpha\delta_\alpha+{\wt{\eta}}^\alpha\mathcal{V}_\alpha$ and $\bar{\sigma}=\bar {\sigma}^\alpha \widehat{e_\alpha}$, we can express (\ref{shahin}) and (\ref{eshgh}) as follow
\begin{equation}
\widetilde{\mathcal{L}}_{\wt{\xi}} F={{\wt{\xi}}}^\alpha(\rho^{i}_{\alpha}\circ\pi)\dfrac{\partial F}{\partial \textbf{x}^i}+({{\wt{\xi}}}^\alpha{\mathcal{B}}^{\gamma}_{\alpha}+{\wt{\eta}}^\gamma)\dfrac{\partial F}{\partial {\textbf{y}}^\gamma};
\end{equation}
\begin{align}\label{jafari}
\widetilde{\mathcal{L}}_{\wt{\xi}}\bar{\sigma}&=\mathcal{V}\llb{{\wt{\xi}}}^\alpha\delta_\alpha+{\wt{\eta}}^\alpha\mathcal{V}_\alpha,
\bar{\sigma}^\beta i(\widehat{e_\beta})\rrb\\ \nonumber
&=\lbrace{{\wt{\xi}}}^\alpha(\rho^{i}_{\alpha}\circ\pi)\dfrac{\partial \bar{\sigma}^\gamma}{\partial \textbf{x}^i}-{{\wt{\xi}}}^\alpha \bar{\sigma}^\beta\dfrac{\partial \mathcal{B}^{\gamma}_{\alpha}}{\partial {\textbf{y}}^\beta}\\ \nonumber
&+{{\wt{\xi}}}^\alpha\mathcal{B}^{\beta}_{\alpha}\dfrac{\partial \bar{\sigma}^\gamma}{\partial {\textbf{y}}^\beta}+{\wt{\eta}}^\alpha\dfrac{\partial \bar{\sigma}^\gamma}{\partial {\textbf{y}}^\alpha}-\bar{\sigma}^\beta\dfrac{\partial {\wt{\eta}}^\gamma}{\partial {\textbf{y}}^\beta}\rbrace\widehat{e_\gamma}.
\end{align}
For future computations, we need the following lemma.
\begin{lemma}
If $\eta$ be a section of $E$, Then
\begin{equation}
\widetilde{\mathcal{L}}_{\eta^C}\widehat{\xi}=\widehat{\llb \eta,\xi\rrb}=\widehat{\mathcal{L}_\eta \xi};
\end{equation}
\begin{equation}\label{sufi}
\widetilde{\mathcal{L}}_{\eta^V}\bar{\xi}=\nabla^{v}_{\widehat{\eta}}\bar{\xi}.
\end{equation}
Moreover, if $\mathcal{H}$ is an Ehresmann connection on $\pounds^\pi E$, then
\begin{equation}\label{Gav}
\widetilde{\mathcal{L}}_{\eta^h}\bar{\xi}=\nabla^{h}_{\widehat{\eta}}\bar{\xi}.
\end{equation}
\end{lemma}
\begin{proof}
Let $\eta=\eta^\alpha e_\alpha$ and $\xi=\xi^\beta e_\beta$. Then using the second equation of (\ref{esi}) and (\ref{jafari}), we obtain
$$\widetilde{\mathcal{L}}_{\eta^C}\widehat{\xi}=\{(\eta^\beta\rho^i_\beta\dfrac{\partial \xi^\alpha}{\partial x^i}-\xi^\beta\rho^i_\beta\dfrac{\partial \eta^\alpha}{\partial x^i}+\eta^\gamma \xi^\beta L^{\alpha}_{\gamma\beta})\circ\pi\rbrace\widehat{e}_\alpha=\widehat{\llb \eta,\xi\rrb}.$$
To proof (\ref{sufi}), we let $\bar{\xi}=\bar{\xi}^\alpha\widehat{e}_\alpha$. Then using the first equation of (\ref{esi}) and (\ref{jafari}) yield
$$
\widetilde{\mathcal{L}}_{\eta^V}\bar{\xi}=(\eta^\beta\circ\pi)\frac{\partial\bar{\xi}^\alpha}{\partial\textbf{y}^\beta}\widehat{e}_\alpha=\nabla^v_{\widehat{\eta}}\bar{\xi}.
$$
Now, we prove (\ref{Gav}). (\ref{Gav1}) and (\ref{jafari}) imply that
$$
\widetilde{\mathcal{L}}_{\eta^h}\bar{\xi}=\{((\eta^\beta\rho^i_\beta)\circ\pi)\frac{\partial\bar{\xi}^\alpha}{\partial\textbf{x}^i}-(\eta^\gamma\circ\pi)\bar{\xi}^\beta\frac{\partial B^\alpha_\gamma}{\partial\textbf{y}^\beta}+(\eta^\gamma\circ\pi)B^\beta_\gamma\frac{\partial\bar{\xi}^\alpha}{\partial\textbf{y}^\beta}\}\widehat{e}_\alpha=\nabla^h_{\widehat{\eta}}\bar{\xi}.
$$
\end{proof}
Thus we can write
$$\widetilde{\mathcal{L}}_{\eta^V}=\nabla^{v}_{\widehat{\eta}}, \qquad \widetilde{\mathcal{L}}_{\eta^h}=\nabla^{h}_{\widehat{\eta}}.$$
\begin{cor}
If $\eta$ be a section of $E$, Then
\begin{equation}\label{flat}
\widetilde{\mathcal{L}}_{\eta^C}\circ j=j\circ \mathcal{L}^{\flat}_{\eta^C},
\end{equation}
\begin{equation}\label{jacobi}
\widetilde{\mathcal{L}}_{\eta^C}\circ\nabla^{v}_{\widehat{\xi}}-\nabla^{v}_{\widehat{\xi}}\circ
\widetilde{\mathcal{L}}_{\eta^C}=\widetilde{\mathcal{L}}_{[\eta,\xi]^V},
\end{equation}
Further if $\mathcal{H}$ be an Ehresmann connection, then
\begin{equation}\label{ss}
\widetilde{\mathcal{L}}_{\eta^C}\circ\nabla^{h}_{\widehat{\xi}}-\nabla^{h}_{\widehat{\xi}}\circ
\widetilde{\mathcal{L}}_{\eta^C}=\widetilde{\mathcal{L}}_{\llb \eta^C,\xi^h\rrb},
\end{equation}
where $\mathcal{L}^{\flat}_{\wt{\xi}}{\wt{\eta}}$ denotes the $\llb\wt{\xi}, \wt{\eta} \rrb$.
\end{cor}
\begin{proof}
For any $\wt{\eta}\in \pounds^\pi(E)$ we have
\begin{equation}\label{XC}
J\llb \eta^C, \wt{\eta}\rrb=\llb \eta^C, J\wt{\eta}\rrb.
\end{equation}
Thus
$$\widetilde{\mathcal{L}}_{\eta^C}\circ j\wt{\eta}=\mathcal{V}\llb \eta^C,J\wt{\eta}\rrb=\mathcal{V}J\llb \eta^C, \wt{\eta}\rrb=j\llb \eta^C, \wt{\eta}\rrb=j\mathcal{L}^{\flat}_{\eta^C}\wt{\eta},$$
proving (\ref{flat}).
For any $F\in C^{\infty}(E)$ and $\bar{\sigma}\in \Gamma(\pounds^\pi E)$,
\begin{align*}
(\widetilde{\mathcal{L}}_{\eta^C}\circ\nabla^{v}_{\widehat{\xi}}-\nabla^{v}_{\widehat{\xi}}\circ
\widetilde{\mathcal{L}}_{\eta^C})F&=\widetilde{\mathcal{L}}_{\eta^C}(\rho_\pounds(\xi^V)F)-\nabla^{v}_{\widehat{\xi}}(\rho_\pounds(\eta^C)F)\\
&=\rho_\pounds(\eta^C)(\rho_\pounds(\xi^V)F)-\rho_\pounds(\xi^V)(\rho_\pounds(\eta^C)F)\\
&=\rho_\pounds(\llb \eta^C,\xi^V\rrb)F=\rho_\pounds([\eta,\xi]^V)F=\widetilde{\mathcal{L}}_{[\eta,\xi]^V}F.
\end{align*}
Using local coordinates and straightforward computations, yield the following
\begin{equation}\label{Gav20}
\llb \eta^V,i\bar{\xi}\rrb=J\llb \eta^V,\mathcal{H}\bar{\xi}\rrb.
\end{equation}
Also from (\ref{aa}), (\ref{flat}), (\ref{XC}),(\ref{Gav20}) and Jacobi identity, we conclude
\begin{align*}
i\circ(\widetilde{\mathcal{L}}_{\eta^C}\circ\nabla^{v}_{\widehat{\xi}}-\nabla^{v}_{\widehat{\xi}}\circ\widetilde
{\mathcal{L}}_{\eta^C})(\bar{\sigma})&=i(\widetilde{\mathcal{L}}_{\eta^C}j\llb \xi^V, \mathcal{H}\bar{\sigma}\rrb)-i(\nabla^{v}_{\widehat{\xi}}i^{-1}\llb \eta^C, i\bar{\sigma}\rrb)\\
&=J\llb \eta^C, \llb \xi^V,\mathcal{H}\bar{\sigma}\rrb\rrb-J\llb \xi^V,\mathcal{H}\circ i^{-1} \llb \eta^C,i\bar{\sigma}\rrb\rrb\qquad\\
&=\llb \eta^C, \llb \xi^V,i\bar{\sigma}\rrb\rrb-\llb \xi^V,\llb \eta^C,i\bar{\sigma}\rrb\rrb\\
&=\llb \eta^C, \llb \xi^V,i\bar{\sigma}\rrb\rrb+\llb \xi^V,\llb i\bar{\sigma}, \eta^C\rrb\rrb\\
&=-\llb i\bar{\sigma},\llb \eta^C,\xi^V\rrb\rrb=\llb[\eta,\xi]^V, i\bar{\sigma}\rrb.
\end{align*}
Hence
$$(\widetilde{\mathcal{L}}_{\eta^C}\circ\nabla^{v}_{\widehat{\xi}}-\nabla^{v}_{\widehat{\xi}}\circ\widetilde
{\mathcal{L}}_{\eta^C})(\bar{\sigma})=i^{-1}\llb[\eta,\xi]^V, i\bar{\sigma}\rrb=\widetilde{\mathcal{L}}_{\llb \eta^C,\xi^h\rrb}\bar{\sigma},$$
proving (\ref{jacobi}). To proving (\ref{ss}), first take a $F\in C^{\infty}(E)$. Then we have $\widetilde{\mathcal{L}}_{\llb \eta^C,\xi^h\rrb}F=\rho_\pounds(\llb \eta^C,\xi^h\rrb)F$ and
\begin{align*}
(\widetilde{\mathcal{L}}_{\eta^C}\circ\nabla^{h}_{\widehat{\xi}}-\nabla^{h}_{\widehat{\xi}}\circ
\widetilde{\mathcal{L}}_{\eta^C})F&=\widetilde{\mathcal{L}}_{\eta^C}(\rho_\pounds(\xi^h)F)-\nabla^{h}_{\widehat{\xi}}(\rho_\pounds(\eta^C)F)\\
&=\rho_\pounds(\eta^C)\rho_\pounds(\xi^h)F-\rho_\pounds(\xi^h)\rho_\pounds(\eta^C)F\\
&=[\rho_\pounds(\eta^C),\rho_\pounds(\xi^h)]F\\
&=\rho_\pounds(\llb \eta^C,\xi^h\rrb) F.
\end{align*}
Hence both sides of (\ref{ss}), act on the functions in the same way. On the other hand
\begin{align*}
i\circ(\widetilde{\mathcal{L}}_{\eta^C}\circ\nabla^{h}_{\widehat{\xi}}-\nabla^{h}_{\widehat{\xi}}\circ
\widetilde{\mathcal{L}}_{\eta^C})(\bar{\sigma})&=i(\widetilde{\mathcal{L}}_{\eta^C}\mathcal{V}\llb \xi^h,i\bar{\sigma}\rrb)-i\nabla^{h}_{\widehat{\xi}}i^{-1}\llb \eta^C,i\bar{\sigma}\rrb\\
&=\llb \eta^C,i\mathcal{V}\llb \xi^h,i\bar{\sigma}\rrb\rrb-i\mathcal{V}\llb \xi^h,\llb \eta^C,i\bar{\sigma}\rrb\rrb\\
&=\llb \eta^C,v\llb \xi^h,i\bar{\sigma}\rrb\rrb-v\llb \xi^h,\llb \eta^C,i\bar{\sigma}\rrb\rrb.
\end{align*}
But $\llb \xi^h,i\bar{\sigma}\rrb$ and $\llb \xi^h,\llb \eta^C,i\bar{\sigma}\rrb\rrb$ are vertical. Thus
\begin{align*}
i\circ(\widetilde{\mathcal{L}}_{\eta^C}\circ\nabla^{h}_{\widehat{\xi}}-\nabla^{h}_{\widehat{\xi}}\circ
\widetilde{\mathcal{L}}_{\eta^C})(\bar{\sigma})&=\llb \eta^C,\llb \xi^h,i\bar{\sigma}\rrb\rrb-\llb \xi^h,\llb \eta^C,i\bar{\sigma}\rrb\rrb\\
&=\llb \llb \eta^C,\xi^h\rrb,i\bar{\sigma}\rrb.
\end{align*}
The assertion will be proved with acting $i^{-1}$ on both sides of the above equality.
\end{proof}
Acting the equality (\ref{jacobi}) on $\widehat{\sigma}$, yields the following.
\begin{cor}\label{me}
Let $\eta,\xi,\sigma\in \Gamma(E)$, then $\nabla^{v}_{\widehat{\xi}}\circ
\widetilde{\mathcal{L}}_{\eta^C}(\widehat{\sigma})=0$.
\end{cor}

\section{Curvature collineation}
In this section, we introduce some tensors that are important in studying the configurations of the bundle maps and yield some results on them in view of collineation.

The \textit{Jacobi endomorphism} $\mathpzc{K}:\mathcal{T}_1(\pi^*\pi)\rightarrow\mathcal{T}_{1}^{1}(\pi^*\pi)$ whrere for example $\mathcal{T}_1$ denoted the $(1, 0)$ tensors of bundle $\pi^*\pi$; is defined as
 \begin{equation}
 \mathpzc{K}(\bar{\eta}):=\mathcal{V}\llb S,\mathcal{H}(\bar{\eta})\rrb.
 \end{equation}
 \begin{cor}
 The Jacobi endomorphism $\mathpzc{K}$ has the following locally expression.
 \begin{align*}
\mathpzc{K}(\widehat{e_\gamma})&=\lbrace-y^\beta(L_{\beta\gamma}^{\theta}\circ\pi)\mathcal{B}^{\alpha}_{\theta} +\mathcal{B}^{\beta}_{\gamma}\mathcal{B}^{\alpha}_{\beta}+
y^\beta(\rho^{i}_\beta\circ\pi)\dfrac{\partial \mathcal{B}^{\alpha}_{\gamma}}{\partial \textbf{x}^i}\\
&-(\rho^{i}_\gamma\circ\pi)\dfrac{\partial S^\alpha}{\partial \textbf{x}^i}+S^\beta\dfrac{\partial \mathcal{B}^{\alpha}_{\gamma}}{\partial {\textbf{y}}^\beta}-\mathcal{B}^{\beta}_{\gamma}\dfrac{\partial S^\alpha}{\partial {\textbf{y}}^\beta}\rbrace\widehat{e_\alpha}.
\end{align*}
 \end{cor}
 \begin{proof}
 In \cite{peyghan}, it is shown that
\begin{align}\label{hlocal}
\mathcal{H}(\bar{\eta})&=\bar{\eta}^\alpha\mathcal{X}_\alpha+\bar{\eta}^\alpha\mathcal{B}^{\beta}_{\alpha}\mathcal{V}_\beta, \qquad \bar{\eta}=\bar{\eta}^\gamma\widehat{e_\gamma}.
\end{align}
Therefore, choosing $\bar{\sigma}=\bar{\sigma}^\gamma\widehat{e_\gamma}$ yields
\begin{align*}
\llb S,\mathcal{H}(\bar{\sigma})\rrb&=y^\alpha \bar{\sigma}^\gamma(L_{\alpha\gamma}^{\theta}\circ\pi)\mathcal{X}_\theta+y^\alpha[(\rho^{i}
_\alpha\circ\pi)\dfrac{\partial \bar{\sigma}^\gamma}{\partial {\textbf{x}}^i}]\mathcal{X}_\gamma+y^\alpha[(\rho^{i}
_\alpha\circ\pi)\dfrac{\partial (\bar{\sigma}^\gamma\mathcal{B}^{\beta}_\gamma)}{\partial \textbf{x}^i}]\mathcal{V}_\beta\\
&-\bar{\sigma}^\gamma\mathcal{B}^{\alpha}_\gamma\mathcal{X}_\alpha+S^\alpha\dfrac{\partial \bar{\sigma}^\gamma}{\partial {\textbf{y}}^\alpha}\mathcal{X}_\gamma-\bar{\sigma}^\gamma[(\rho^{i}_\gamma\circ\pi)\dfrac{\partial S^\alpha}{\partial \textbf{x}^i}]\mathcal{V}_\alpha\\
&+S^\alpha\dfrac{\partial (\bar{\sigma}^\gamma\mathcal{B}^{\beta}_\gamma)}{\partial{\textbf{y}}^\alpha}\mathcal{V}_\beta-\bar{\sigma}^\gamma\mathcal{B}^{\beta}_\gamma\dfrac{\partial S^\alpha}{\partial {\textbf{y}}^\beta}\mathcal{V}_\alpha.
\end{align*}
Now, setting $\bar{\sigma}=\widehat{e_\gamma}$ and using $\mathcal{H}(\widehat{e_\gamma})=\delta_\gamma$ arise
\begin{align*}
\llb S,\mathcal{H}(\widehat{e_\gamma})\rrb=\llb S,\delta_\gamma\rrb&=\lbrace y^\beta(L_{\beta\gamma}^{\alpha}\circ\pi)-\mathcal{B}^{\alpha}_{\gamma}\rbrace\mathcal{X}_\alpha+\lbrace
y^\beta(\rho^{i}_\beta\circ\pi)\dfrac{\partial \mathcal{B}^{\alpha}_{\gamma}}{\partial \textbf{x}^i}\\
&-(\rho^{i}_\gamma\circ\pi)\dfrac{\partial S^\alpha}{\partial \textbf{x}^i}+S^\beta\dfrac{\partial \mathcal{B}^{\alpha}_{\gamma}}{\partial \textbf{y}^\beta}-\mathcal{B}^{\beta}_{\gamma}\dfrac{\partial S^\alpha}{\partial \textbf{y}^\beta}\rbrace\mathcal{V}_\alpha.
\end{align*}
Acting $\mathcal{V}$ on the above equality, yields the assertion.
\end{proof}
Using the Jacobi endomorphism $\mathcal{K}$, the \textit{fundamental affine curvature} $\mathpzc{R}:\mathcal{T}_{2}^{0}(\pi^*\pi)\rightarrow\mathcal{T}_{2}^{1}(\pi^*\pi)$ is defined by
 \begin{equation}
 \mathpzc{R}(\bar{\eta},\bar{\xi}):=\dfrac{1}{3}(\nabla^v\mathcal{K}(\bar{\xi},\bar{\eta})-\nabla^v\mathcal{K}(\bar{\eta},\bar{\xi})).
 \end{equation}

 The \textit{affine curvature} $\mathpzc{H}:\mathcal{T}_{3}^{0}(\pi^*\pi)\rightarrow\mathcal{T}_{3}^{1}(\pi^*\pi)$ is defined by
 \begin{equation}
 \mathpzc{H}(\bar{\eta},\bar{\xi})\bar{\sigma}:=\nabla^v\mathpzc{R}(\bar{\sigma},\bar{\eta},\bar{\xi}).
 \end{equation}
 For a $A\in \mathcal{T}_{l+1}^{1}(\pi)$, its \textit{trace} denoted by $\text{tr}(A)$ is defined as follows
\begin{equation*}
\text{tr}(A)({X_1}, ...,{X_l}):=\text{tr}(\Phi), \qquad \Phi({Z}):=A({Z},{X_1}, ...,{X_l}),
\end{equation*}
where $\pi$ is a bundle projection.

The \textit{projective deviation tensor} $\mathpzc{W}^\circ$ is defined by
\begin{equation}
\mathpzc{W}^\circ:=\mathpzc{K}-\dfrac{1}{n-1}(\text{tr}\mathpzc{K})Id_{\Gamma({\pounds^\pi E})}+\dfrac{3}{n+1}(\text{tr}\mathpzc{R})\otimes\delta+\dfrac{2-n}{n^2-1}(\nabla^v\text{tr}\mathpzc{K})\otimes\delta.
\end{equation}
Note that we can quickly rewrite this tensor as follows
\begin{equation}
\mathpzc{W}^\circ=\mathpzc{K}-\stackrel{\circ}{\mathpzc{K}}Id_{\Gamma({\pounds^\pi E})}+\dfrac{1}{n+1}(\nabla^v\stackrel{\circ}{\mathpzc{K}}-\text{tr}\nabla^v\mathpzc{K})\otimes\delta,
\end{equation}
where $\stackrel{\circ}{\mathpzc{K}}:=\dfrac{1}{n-1}\text{tr}\mathpzc{K}$.

The \textit{fundamental projective curvature} is defined by
\begin{equation}
 \mathpzc{W}(\bar{\eta},\bar{\xi}):=\dfrac{1}{3}(\nabla^v\mathpzc{W}^\circ(\bar{\xi},\bar{\eta})-\nabla^v\mathpzc{W}^\circ(\bar{\eta},\bar{\xi})),
\end{equation}
and the \text{projective curvature} $W^*$ as
\begin{equation}
\mathpzc{W}^*(\bar{\eta},\bar{\xi})(\bar{\sigma}):=\nabla^v\mathpzc{W}(\bar{\sigma},\bar{\eta},\bar{\xi}).
\end{equation}
The \textit{Berwald curvature} $\mathpzc{B}$ and \textit{Douglas curvature} $\mathpzc{D}$ are defined by
\begin{equation}
\mathpzc{B}(\widehat{\eta},\widehat{\xi})(\widehat{\sigma}):=(\nabla^v\nabla^h\widehat{\sigma})(\widehat{\eta},\widehat{\xi}),
\end{equation}
and
\begin{equation}
\mathpzc{D}:=\mathpzc{B}-\dfrac{1}{n+1}\lbrace(\text{tr} \mathpzc{B})\odot Id_{\Gamma({\pounds^\pi E})}+\nabla^v\text{tr}\mathpzc{B}\otimes\delta\rbrace,
\end{equation}
where $\odot$ shows the  numerical factor is omitted from symmetric product.

 A {\it{Lie symmetry}} of the semispray $S$ is a section $\eta$ of $E$ such that $\llb S, \eta^C\rrb=0$. We have the following.
\begin{proposition}
A section $\eta=\eta^\alpha e_\alpha$ of $E$ is a Lie symmetry of $S$ if and only if
\begin{equation}\label{khune}
\textbf{y}^\beta \textbf{y}^\lambda(\rho^i_\lambda\circ\pi)\frac{\partial(\eta^\alpha_{|_\beta}\circ\pi)}{\partial\textbf{x}^i}-((\eta^\lambda\rho^i_\lambda)\circ\pi)\frac{\partial S^\alpha}{\partial\textbf{x}^i}+S^\lambda(\eta^\alpha_{|_\lambda}\circ\pi)-y^\beta(\eta^\lambda_{|_\beta}\circ\pi)\frac{\partial S^\alpha}{\partial\textbf{y}^\lambda}=0,
\end{equation}
where $\eta^\alpha_{|_\beta}:=\rho^j_\beta\frac{\partial \eta^\alpha}{\partial x^j}-\eta^\gamma L^\alpha_{\gamma\beta}$.
\end{proposition}
\begin{proof}
Using (\ref{esi}) and (\ref{semispray}) we obtain
\begin{align*}
\llb S, \eta^C\rrb&=\llb\textbf{y}^\alpha\mathcal{X}_\alpha+S^\alpha\mathcal{V}_\alpha, (\eta^\lambda\circ\pi)\mathcal{X}_\lambda+\textbf{y}^\beta(\eta^\lambda_{|_\beta}\circ\pi)\mathcal{V}_\lambda\rrb\\
&=\{\textbf{y}^\lambda\rho^i_\lambda\frac{\partial(\eta^\alpha\circ\pi)}{\partial\textbf{x}^i}+\textbf{y}^\sigma((\eta^\lambda L^\alpha_{\sigma\lambda})\circ\pi)-\textbf{y}^\beta(\eta^\alpha_{|_\beta}\circ\pi)\}\mathcal{X}_\alpha\\
&\ \ +\{\textbf{y}^\beta \textbf{y}^\lambda(\rho^i_\lambda\circ\pi)\frac{\partial(\eta^\alpha_{|_\beta}\circ\pi)}{\partial\textbf{x}^i}-((\eta^\lambda\rho^i_\lambda)\circ\pi)\frac{\partial S^\alpha}{\partial\textbf{x}^i}+S^\lambda(\eta^\alpha_{|_\lambda}\circ\pi)\\
&\ \ -y^\beta(\eta^\lambda_{|_\beta}\circ\pi)\frac{\partial S^\alpha}{\partial\textbf{y}^\lambda}\}\mathcal{V}_\alpha.
\end{align*}
Using direct calculation we deduce that the coefficient of $\mathcal{X}_\alpha$ vanishes. Therefore $\llb S, \eta^C\rrb=0$ if and only if the coefficient of $\mathcal{V}_\alpha$ is zero.
\end{proof}
\begin{lemma}\label{lemedoktor}
Let $S$ be a spray on the prolongation of a Lie algebroid with the structure $(\llb ., .\rrb, \rho_\pounds)$ and $\eta\in\Gamma(E)$. Then the following statements are equivalent.
\begin{itemize}
\item[(i)] $\eta$ is a Lie symmetry of $S$;
\item[(ii)] $[\eta,\xi]^h=\llb \eta^C, \xi^h\rrb$ \text{for any} $\xi\in\Gamma(E)$;
\item[(iii)] $[ v, \eta^C]^{F-N}=0$.
\end{itemize}
\end{lemma}
\begin{proof}
$(i)\Longrightarrow(ii)$: Setting
\begin{equation}\label{sobh}
A(\eta,\xi):=[\eta,\xi]^h-\llb \eta^C, \xi^h\rrb,
\end{equation}
it arises the following
\begin{align}\label{man}
A(\eta,\xi)=\xi^\beta\lbrace \eta^\lambda\rho^i_\lambda\dfrac{\partial \mathcal{B}^\alpha_\beta}{\partial \textbf{x}^i}-{\textbf{y}}^\gamma\rho^i_\beta\dfrac{\partial \eta^\alpha_{|\gamma}}{\partial\textbf{x}^i}+{\textbf{y}}^\gamma \eta^\lambda_{|\gamma}\dfrac{\partial \mathcal{B}^\alpha_\beta}{\partial {\textbf{y}}^\lambda}- \mathcal{B}^\lambda_\beta \eta^\alpha_{|\lambda}+\eta^\gamma_{|\beta} \mathcal{B}^\alpha_\gamma\rbrace\mathcal{V}_\alpha,
\end{align}
where $\eta=\eta^\alpha e_\alpha$ and $\xi=\xi^\beta e_\beta$.
Putting (\ref{mokh}) in the above equation, yields
\begin{align}\label{yasi}
A(\eta,\xi)&=\xi^\beta\lbrace \dfrac{1}{2}\eta^\lambda\rho^i_\lambda(\dfrac{\partial^2S^\alpha}{\partial {\textbf{x}}^i\partial{\textbf{y}}^\beta}-{\textbf{y}}^\theta\dfrac{\partial L^{\alpha}_{\beta\theta}}{\partial{\textbf{x}}^i})-{\textbf{y}}^\gamma\rho^i_\beta\dfrac{\partial \eta^\alpha_{|\gamma}}{\partial {\textbf{x}}^i}\nonumber\\
&+\dfrac{1}{2}{\textbf{y}}^\gamma \eta^\lambda_{|\gamma}(\dfrac{\partial^2 S^\alpha}{\partial{\textbf{y}}^\beta\partial{\textbf{y}}^\lambda}-L^\alpha_{\beta\lambda})-\dfrac{1}{2}\eta^\alpha_{|\lambda}(\dfrac{\partial S^\lambda}{\partial {\textbf{y}}\beta}-{\textbf{y}}^\theta L^\lambda_{\beta\theta})\nonumber\\
&+\dfrac{1}{2}\eta^\gamma_{|\beta}(\dfrac{\partial S^\alpha}{\partial {\textbf{y}}\gamma}-{\textbf{y}}^\theta L^\alpha_{\gamma\theta})\rbrace.
\end{align}
Differentiating from (\ref{khune}) with respect to $\textbf{y}$ and putting it into the (\ref{yasi}), we obtain
\begin{align}
A(\eta,\xi)&=\dfrac{1}{2}\xi^\beta\lbrace {\textbf{y}}^\lambda\rho^i_\lambda\dfrac{\partial \eta^\alpha_{|\beta}}{\partial {\textbf{x}}^i}-\eta^\gamma_{|\beta}{\textbf{y}}^\lambda L^\alpha_{\gamma\lambda}-\eta^\lambda\rho^i_\lambda{\textbf{y}}^\gamma\dfrac{\partial L^\alpha_{\beta\gamma}}{\partial{\textbf{x}}^i}\nonumber\\
&-{\textbf{y}}^\gamma \eta^\lambda_{|\gamma}L^\alpha_{\beta\lambda}+\eta^\alpha_{|\lambda}{\textbf{y}}^\gamma L^\lambda_{\beta\gamma}-{\textbf{y}}^\gamma\rho^i_\beta\dfrac{\partial \eta^\alpha_{|\gamma}}{\partial {\textbf{x}}^i}\rbrace.
\end{align}
Putting $\eta^\alpha_{|\beta}=\rho^j_\beta\dfrac{\partial \eta^\alpha}{\partial \textbf{x}^i}-\eta^\gamma L^\alpha_{\beta\gamma}$ into the above equation, simplification and suitable changing of indices, the following will be yield.
\begin{align}
A(\eta,\xi)&=\dfrac{1}{2}\xi^\beta\lbrace{\textbf{y}^\lambda}\dfrac{\partial \eta^\alpha}{\partial{\textbf{x}}^j}(\rho^i_\lambda\dfrac{\partial \rho^j_\beta}{\partial{\textbf{x}}^i}-\rho^i_\beta\dfrac{\partial \rho^j_\lambda}{\partial{\textbf{x}}^i}+\rho^j_\gamma L^\gamma_{\beta\lambda})\nonumber\\
&-{\textbf{y}}^\lambda \eta^\gamma(\rho^i_\lambda\dfrac{\partial L^\alpha_{\gamma\beta}}{\partial{\textbf{x}}^i}
+L^\theta_{\gamma\beta}L^\alpha_{\lambda\theta}+\rho^i_\gamma\dfrac{\partial L^\alpha_{\beta\lambda}}{\partial {\textbf{x}}^i}+L^\theta_{\lambda\gamma}L^\alpha_{\beta\theta}\nonumber\\
&+L^\alpha_{\gamma\theta}L^\theta_{\beta\lambda}+\rho^i_\beta\dfrac{\partial L^\alpha_{\lambda\gamma}}{\partial{\textbf{x}}^i})\rbrace.
\end{align}
Applying (\ref{2}) with above equation, clearly $A(\eta,\xi)=0$.\\
$(ii)\Longrightarrow(iii)$: Direct calculations give us $[v, \eta^C]^{F-N}(\mathcal{V}_\beta)=0$ and $[v, \eta^C]^{F-N}(\mathcal{X}_\beta)=A(\eta, \delta_\beta)$. Since $(ii)$ holds, thus from (\ref{sobh}), we have $A(\eta,\xi)=0$ for any $\xi\in\Gamma(E)$. Thus $A(\eta,\delta_\alpha)=0$ and consequently $(iii)$ is hold.\\
$(iii)\Longrightarrow(i)$: From $(iii)$ we obtain
\begin{align}\label{hamid}
 \eta^\lambda\rho^i_\lambda\dfrac{\partial \mathcal{B}^\alpha_\beta}{\partial {\textbf{x}}^i}-{\textbf{y}}^\gamma\rho^i_\beta\dfrac{\partial \eta^\alpha_{|\gamma}}{\partial{\textbf{x}}^i}+{\textbf{y}}^\gamma \eta^\lambda_{|\gamma}\dfrac{\partial \mathcal{B}^\alpha_\beta}{\partial {\textbf{y}}^\lambda}- \mathcal{B}^\lambda_\beta \eta^\alpha_{|\lambda}+\eta^\gamma_{|\beta} \mathcal{B}^\alpha_\gamma=0.
\end{align}
Plugging (\ref{mokh}) into the (\ref{hamid}), relation (\ref{khune}) will be yield. Thus $(i)$ is hold.
\end{proof}
A projectable section $\wt{\xi}$ is said to be a \textit{curvature collineation of a curvature tensor} $C\in \mathcal{T}^{1}_{k}$ that $k\in\lbrace1,2,3\rbrace$ of a spray algebroid, if $\widetilde{\mathcal{L}}_{\wt{\xi}}C=0$.
\begin{theorem}\label{similar}
Let $S$ be a spray on the prolongation of a Lie algebroid with the structure $(\llb ., .\rrb, \rho_\pounds)$ and $\eta$ be a Lie symmetry of $S$. Then $\eta^C$ is a curvature collineation of
$\mathpzc{K}$, $\mathpzc{R}$ and $\mathpzc{H}$.
\end{theorem}
\begin{proof}
Let $\xi$ be a section of $E$. Then
\begin{align*}
(\widetilde{\mathcal{L}}_{\eta^C}\mathpzc{K})\widehat{\xi}&=\widetilde{\mathcal{L}}_{\eta^C}(\mathpzc{K}(\widehat{\xi}))-\mathpzc{K}(\widetilde{\mathcal{L}}_{\eta^C}\widehat{\xi})=\widetilde{\mathcal{L}}_{\eta^C}(\mathcal{V}\llb S,\xi^h\rrb)-\mathcal{V}\llb S,\mathcal{H}(\widetilde{\mathcal{L}}_{\eta^C}\widehat{\xi})\rrb\\
&=i^{-1}\llb \eta^C,v\llb S,\xi^h\rrb\rrb-\mathcal{V}\llb S, \mathcal{H}\circ i^{-1}\llb \eta^C,\xi^v\rrb\rrb\\
&=i^{-1}(\llb \eta^C,v\llb S,\xi^h\rrb\rrb-v\llb S,[\eta,\xi]^h\rrb\rrb).
\end{align*}
Using lemma (\ref{lemedoktor}), yields
$$(\widetilde{\mathcal{L}}_{\eta^C}\mathpzc{K})\widehat{\xi}=i^{-1}(\llb \eta^C,v\llb S,\xi^h\rrb\rrb-v\llb S,\llb \eta^C,\xi^h\rrb\rrb).$$
According to the Jacobi identity and because of $\eta$ is a Lie symmetry of $S$, we have $\llb S,\llb \eta^C,\xi^h\rrb\rrb=\llb \eta^C,\llb S,\xi^h\rrb\rrb$. Therefore and using lemma (\ref{lemedoktor}), the following will be yield.
$$(\widetilde{\mathcal{L}}_{\eta^C}\mathpzc{K})\widehat{\xi}=-i^{-1}([v,\eta^C]^{F-N}\llb S,\xi^h\rrb)=0,$$
proving the first assertion. Also the vanishing of $\widetilde{\mathcal{L}}_{\eta^C}\mathpzc{K}$ is equivalent to
\begin{equation}\label{mm}
\widetilde{\mathcal{L}}_{\eta^C}\circ \mathpzc{K}=\mathpzc{K}\circ\widetilde{\mathcal{L}}_{\eta^C}.
\end{equation}
Using relation (\ref{jacobi}), we deduce
\begin{equation}
\widetilde{\mathcal{L}}_{\eta^C}(\nabla^{v}\mathpzc{K}(\widehat{\sigma},\widehat{\xi}))=\nabla^{v}_{\widehat{\sigma}}\circ
\widetilde{\mathcal{L}}_{\eta^C}\mathpzc{K}(\widehat{\xi})+
\widetilde{\mathcal{L}}_{[\eta,\sigma]^V}(\mathpzc{K}(\widehat{\xi})).
\end{equation}
On proceeding to prove the second, we check it on the locally bases as follow
\begin{equation}\label{local}
(\widetilde{\mathcal{L}}_{\eta^C}\mathpzc{R})(\widehat{\xi},\widehat{\sigma})=\widetilde{\mathcal{L}}_{\eta^C}(\mathpzc{R}(\widehat{\xi},\widehat{\sigma}))-\mathpzc{R}(\widetilde{\mathcal{L}}_{\eta^C}\widehat{\xi},\widehat{\sigma})-\mathpzc{R}(\widehat{\xi},\widetilde{\mathcal{L}}_{\eta^C}\widehat{\sigma}).
\end{equation}
But
\begin{align*}
\widetilde{\mathcal{L}}_{\eta^C}(\mathpzc{R}(\widehat{\xi},\widehat{\sigma}))&=\dfrac{1}{3}(\widetilde{\mathcal{L}}_{\eta^C}(\nabla^v\mathpzc{K}(\widehat{\sigma},\widehat{\xi}))-\widetilde{\mathcal{L}}_{\eta^C}(\nabla^v\mathpzc{K}(\widehat{\xi},\widehat{\sigma})))\\
&=\dfrac{1}{3}(\nabla^{v}_{\widehat{\sigma}}\circ
\widetilde{\mathcal{L}}_{\eta^C}\mathpzc{K}(\widehat{\xi})+
\widetilde{\mathcal{L}}_{[\eta,\sigma]^V}(\mathpzc{K}(\widehat{\xi}))\\
&-\nabla^{v}_{\widehat{\xi}}\circ
\widetilde{\mathcal{L}}_{\eta^C}\mathpzc{K}(\widehat{\sigma})-
\widetilde{\mathcal{L}}_{[\eta,\xi]^V}(\mathpzc{K}(\widehat{\sigma}))),
\end{align*}
and
\begin{align*}
\mathpzc{R}(\widetilde{\mathcal{L}}_{\eta^C}\widehat{\xi},\widehat{\sigma})&=\dfrac{1}{3}((\nabla^v\mathpzc{K})(\widehat{\sigma},\widetilde{\mathcal{L}}_{\eta^C}\widehat{\xi})-(\nabla^v\mathpzc{K})(\widetilde{\mathcal{L}}_{\eta^C}\widehat{\xi},\widehat{\sigma}))\\
&=\dfrac{1}{3}((\nabla^v_{\widehat{\sigma}}\mathpzc{K})(\widetilde{\mathcal{L}}_{\eta^C}\widehat{\xi})-(\nabla^v_{\widetilde{\mathcal{L}}_{\eta^C}\widehat{\xi}}\mathpzc{K})(\widehat{\sigma}).
\end{align*}
According to corollary (\ref{me}) and relations (\ref{20}) and (\ref{mm}) and noting that $\widetilde{\mathcal{L}}_{\eta^C}\widehat{\xi}=i^{-1}\llb \eta^C,\xi^V\rrb=\widehat{[\eta,\xi]}$, we can derive
\begin{align*}
\mathpzc{R}(\widetilde{\mathcal{L}}_{\eta^C}\widehat{\xi},\widehat{\sigma})&=\dfrac{1}{3}(\nabla^v_{\widehat{\sigma}}\circ\widetilde{\mathcal{L}}_{\eta^C}(\mathpzc{K}(\widehat{\xi}))-\nabla^v_{\widehat{[\eta,\xi]}}(\mathpzc{K}(\widehat{\sigma}))).
\end{align*}
Thus
\begin{align*}
-\lbrace\mathpzc{R}(\widetilde{\mathcal{L}}_{\eta^C}\widehat{\xi},\widehat{\sigma})+\mathpzc{R}(\widehat{\xi},\widetilde{\mathcal{L}}_{\eta^C}\widehat{\sigma})\rbrace&=-\mathpzc{R}(\widetilde{\mathcal{L}}_{\eta^C}\widehat{\xi},\widehat{\sigma})+\mathpzc{R}(\widetilde{\mathcal{L}}_{\eta^C}\widehat{\sigma},\widehat{\xi})\\
&=\dfrac{1}{3}\lbrace-\nabla^v_{\widehat{\sigma}}\circ\widetilde{\mathcal{L}}_{\eta^C}(\mathpzc{K}(\widehat{\xi}))+\nabla^v_{\widehat{[\eta,\xi]}}(\mathpzc{K}(\widehat{\sigma}))\\
&+\nabla^v_{\widehat{\xi}}\circ\widetilde{\mathcal{L}}_{\eta^C}(\mathpzc{K}(\widehat{\sigma}))-\nabla^v_{\widehat{[\eta,\sigma]}}(\mathpzc{K}(\widehat{\xi}))\rbrace\\
&=\dfrac{1}{3}\lbrace-\nabla^v_{\widehat{\sigma}}\circ\widetilde{\mathcal{L}}_{\eta^C}(\mathpzc{K}(\widehat{\xi}))+\widetilde{\mathcal{L}}_{[\eta,\xi]^V}(\mathpzc{K}(\widehat{\sigma}))\\
&+\nabla^v_{\widehat{\xi}}\circ\widetilde{\mathcal{L}}_{\eta^C}(\mathpzc{K}(\widehat{\sigma}))-\widetilde{\mathcal{L}}_{[\eta,\sigma]^V}(\mathpzc{K}(\widehat{\xi}))\rbrace.
\end{align*}
Therefore the right-hand side of (\ref{local}) is zero. Hence we proved the second assertion. Finally for $\mathpzc{H}$, using (\ref{jacobi}) and corollary (\ref{me}) possess to
\begin{align*}
(\widetilde{\mathcal{L}}_{\eta^C}\mathpzc{H})(\widehat{\xi},\widehat{\sigma},\widehat{\vartheta})&=\widetilde{\mathcal{L}}_{\eta^C}(\mathpzc{H}(\widehat{\xi},\widehat{\sigma})\widehat{\vartheta})-\mathpzc{H}(\widetilde{\mathcal{L}}_{\eta^C}\widehat{\xi},\widehat{\sigma})\widehat{\vartheta}\\
&-\mathpzc{H}(\widehat{\xi},\widetilde{\mathcal{L}}_{\eta^C}\widehat{\sigma})\widehat{\vartheta}-\mathpzc{H}(\widehat{\xi},\widehat{\sigma})\widetilde{\mathcal{L}}_{\eta^C}\widehat{\vartheta}\\
&=\widetilde{\mathcal{L}}_{\eta^C}(\nabla^{v}_{\widehat{\vartheta}}(\mathpzc{R}(\widehat{\xi},\widehat{\sigma})))-(\nabla^{v}_{\widehat{\vartheta}}\mathpzc{R})(\widetilde{\mathcal{L}}_{\eta^C}\widehat{\xi},\widehat{\sigma})\\
&-(\nabla^{v}_{\widehat{\vartheta}}\mathpzc{R})(\widehat{\xi},\widetilde{\mathcal{L}}_{\eta^C}\widehat{\sigma})-(\nabla^{v}_{\widehat{[\eta,\vartheta]}}\mathpzc{R})(\widehat{\xi},\widehat{\sigma})\\
&=\nabla^{v}_{\widehat{\vartheta}}\circ \widetilde{\mathcal{L}}_{\eta^C}(\mathpzc{R}(\widehat{\xi},\widehat{\sigma}))+\widetilde{\mathcal{L}}_{[\eta,\vartheta]^V}(\mathpzc{R}(\widehat{\xi},\widehat{\sigma}))\\
&-\nabla^{v}_{\widehat{\vartheta}}(\mathpzc{R}(\widetilde{\mathcal{L}}_{\eta^C}\widehat{\xi},\widehat{\sigma}))-\nabla^{v}_{\widehat{\vartheta}}(\mathpzc{R}(\widehat{\xi},\widetilde{\mathcal{L}}_{\eta^C}\widehat{\sigma}))
-\nabla^{v}_{\widehat{[\eta,\vartheta]}}(\mathpzc{R}(\widehat{\xi},\widehat{\sigma})).
\end{align*}
Since $\widetilde{\mathcal{L}}_{\eta^C}\mathpzc{R}=0$, thus from (\ref{local}) we obtain
\begin{equation}\label{lloccal}
\widetilde{\mathcal{L}}_{\eta^C}(\mathpzc{R}(\widehat{\xi},\widehat{\sigma}))=\mathpzc{R}(\widetilde{\mathcal{L}}_{\eta^C}\widehat{\xi},\widehat{\sigma})+\mathpzc{R}(\widehat{\xi},\widetilde{\mathcal{L}}_{\eta^C}\widehat{\sigma}).
\end{equation}
Arising from (\ref{lloccal}) and relation (\ref{sufi}) we deduce that $\widetilde{\mathcal{L}}_{\eta^C}\mathpzc{H}=0$.
\end{proof}
\begin{lemma}\label{ees}
Let $F\in C^{\infty}(E)$, then $(\widetilde{\mathcal{L}}_{\eta^C}\nabla^vF)(\widehat{\xi})=\rho_\pounds(\xi^V)(\rho_\pounds(\eta^C)F)$.
\end{lemma}
\begin{proof}
For every $\xi\in\Gamma(E)$, we have
\begin{align*}
(\widetilde{\mathcal{L}}_{\eta^C}\nabla^vF)(\widehat{\xi})&=\rho_\pounds(\eta^C)(\rho_\pounds(\xi^V)F)-\nabla^vF(\widehat{[\eta,\xi]})\\
&=\rho_\pounds(\llb \eta^C,\xi^V\rrb)F+\rho_\pounds(\xi^V)(\rho_\pounds(\eta^C)F)-\rho_\pounds([\eta,\xi]^V)F\\
&=\rho_\pounds(\xi^V)(\rho_\pounds(\eta^C)F).
\end{align*}
\end{proof}
\begin{theorem}
Let $S$ be a spray on the prolongation of a Lie algebroid with the structure $(\llb ., .\rrb, \rho_\pounds)$ and $\eta$ be a Lie symmetry of $S$. Then $\eta^C$ is a curvature collineation of
$\mathpzc{W^\circ}$, $\mathpzc{W}$ and $\mathpzc{W^*}$.
\end{theorem}
\begin{proof}
Noting that $\widetilde{\mathcal{L}}_{\eta^C}Id_{\Gamma({\pounds^\pi E})}\equiv0$, $\widetilde{\mathcal{L}}_{\eta^C}\delta=\mathcal{V}(0)=0$ and $\widetilde{\mathcal{L}}_{\eta^C}\circ\text{tr}=\text{tr}\circ\widetilde{\mathcal{L}}_{\eta^C}$ we have
\begin{align*}
\widetilde{\mathcal{L}}_{\eta^C}\mathpzc{W^\circ}&=\widetilde{\mathcal{L}}_{\eta^C}\mathpzc{K}-\dfrac{1}{n-1}\text{tr}(\widetilde{\mathcal{L}}_{\eta^C}\mathpzc{K})Id_{\Gamma({\pounds^\pi E})}+\dfrac{3}{n+1}(\text{tr}(\widetilde{\mathcal{L}}_{\eta^C}\mathpzc{R}))\otimes\delta\\
&+\dfrac{2-n}{n^2-1}(\widetilde{\mathcal{L}}_{\eta^C}(\nabla^v\text{tr}\mathpzc{K}))\otimes\delta\\
&=\dfrac{2-n}{n^2-1}(\widetilde{\mathcal{L}}_{\eta^C}(\nabla^v\text{tr}\mathpzc{K}))\otimes\delta.
\end{align*}
Now according to  lemma (\ref{ees}), $(\widetilde{\mathcal{L}}_{\eta^C}(\nabla^v\text{tr}\mathpzc{K}))(\widehat{\xi})=0$, that proves $\widetilde{\mathcal{L}}_{\eta^C}\mathpzc{W^\circ}=0$. The similar result for $\mathpzc{W}$ and $\mathpzc{W}^*$ is analogous to theorem (\ref{similar}).
\end{proof}
\begin{theorem}
Let $S$ be a spray on the prolongation of a Lie algebroid with the structure $(\llb ., .\rrb, \rho_\pounds)$ and $\eta$ be a Lie symmetry of $S$. Then $\eta^C$ is a curvature collineation of the
Berwald curvature.
\end{theorem}
\begin{proof}
For any sections $\xi, \sigma, \vartheta\in\Gamma(E)$,
\begin{align*}
(\widetilde{\mathcal{L}}_{\eta^C}\mathpzc{B})(\widehat{\xi},\widehat{\sigma},\widehat{\vartheta})&=\widetilde{\mathcal{L}}_{\eta^C}(\mathpzc{B}(\widehat{\xi},\widehat{\sigma})\widehat{\vartheta})-\mathpzc{B}(\widetilde{\mathcal{L}}_{\eta^C}\widehat{\xi},\widehat{\sigma})\widehat{\vartheta}\\
&-\mathpzc{B}(\widehat{\xi},\widetilde{\mathcal{L}}_{\eta^C}\widehat{\sigma})\widehat{\vartheta}-\mathpzc{B}(\widehat{\xi},\widehat{\sigma})\widetilde{\mathcal{L}}_{\eta^C}\widehat{\vartheta}\\
&=\widetilde{\mathcal{L}}_{\eta^C}((\nabla^v\nabla^h\widehat{\vartheta})(\widehat{\xi},\widehat{\sigma}))-((\nabla^v\nabla^h\widehat{\vartheta})(\widehat{[\eta,\xi]},\widehat{\sigma}))\\
&-((\nabla^v\nabla^h\widehat{\vartheta})(\widehat{\xi},\widehat{[\eta,\sigma]}-((\nabla^v\nabla^h\widetilde{\mathcal{L}}_{\eta^C}\widehat{\vartheta})(\widehat{\xi},\widehat{\sigma})\\
&=\widetilde{\mathcal{L}}_{\eta^C}(\nabla^v_{\widehat{\xi}}\nabla^h_{\widehat{\sigma}}\widehat{\vartheta})-\nabla^v_{\widehat{[\eta,\xi]}}\nabla^h_{\widehat{\sigma}}
\widehat{\vartheta}-\nabla^v_{\widehat{\xi}}\nabla^h_{\widehat{[\eta,\sigma]}}\widehat{\vartheta}-\nabla^v_{\widehat{\xi}}
\nabla^h_{\widehat{\sigma}}\widetilde{\mathcal{L}}_{\eta^C}\widehat{\vartheta}\\
&=\nabla^v_{\widehat{\xi}}(\widetilde{\mathcal{L}}_{\eta^C}\nabla^h_{\widehat{\sigma}}\widehat{\vartheta})+\widetilde{\mathcal{L}}_{[\eta,\xi]^V}
\nabla^h_{\widehat{\sigma}}\widehat{\vartheta}-\nabla^v_{\widehat{[\eta,\xi]}}\nabla^h_{\widehat{\sigma}}\widehat{\vartheta}\\
&-\nabla^v_{\widehat{\xi}}\nabla^h_{\widehat{[\eta,\sigma]}}\widehat{\vartheta}-\nabla^v_{\widehat{\xi}}\widetilde{\mathcal{L}}
_{\eta^C}\nabla^h_{\widehat{\sigma}}\widehat{\vartheta}+\nabla^v_{\widehat{\xi}}
\widetilde{\mathcal{L}}_{\llb \eta^C,\sigma^h\rrb}\widehat{\vartheta}\\
&=-\nabla^v_{\widehat{\xi}}\nabla^h_{\widehat{[\eta,\sigma]}}\widehat{\vartheta}+\nabla^v_{\widehat{\xi}}
\widetilde{\mathcal{L}}_{\llb \eta^C,\sigma^h\rrb}\widehat{\vartheta}=0,
\end{align*}
proving the assertion.
\end{proof}
\begin{theorem}
Let $S$ be a spray on the prolongation of a Lie algebroid with the structure $(\llb ., .\rrb, \rho_\pounds)$ and $\eta$ be a Lie symmetry of $S$. Then $\eta^C$ is a curvature collineation of the
Douglas curvature.
\end{theorem}
\begin{proof}
It is enough to show that $\widetilde{\mathcal{L}}_{\eta^C}(\nabla^v\text{tr}\mathpzc{B})=0$. For any sections $\xi, \sigma, \vartheta\in\Gamma(E)$,
\begin{align*}
\widetilde{\mathcal{L}}_{\eta^C}(\nabla^v\text{tr}\mathpzc{B})(\widehat{\xi},\widehat{\sigma},\widehat{\vartheta})&=\rho_\pounds(\eta^C)((\nabla^v\text{tr}\mathpzc{B})(\widehat{\xi},\widehat{\sigma},\widehat{\vartheta}))-(\nabla^v\text{tr}\mathpzc{B})(\widehat{[\eta,\xi]},\widehat{\sigma},\widehat{\vartheta})\\
&-(\nabla^v\text{tr}\mathpzc{B})(\widehat{\xi},\widetilde{\mathcal{L}}_{\eta^C}\widehat{\sigma},\widehat{\vartheta})-(\nabla^v\text{tr}\mathpzc{B})(\widehat{\xi},\widehat{\sigma},\widetilde{\mathcal{L}}_{\eta^C}\widehat{\vartheta})\\
&=\rho_\pounds(\eta^C)\rho_\pounds(\xi^V)(\text{tr}\mathpzc{B}(\widehat{\sigma},\widehat{\vartheta}))-\rho_\pounds([\eta,\xi]^V)(\text{tr}\mathpzc{B}(\widehat{\sigma},\widehat{\vartheta}))\\
&-\rho_\pounds(\xi^V)(\text{tr}\mathpzc{B}(\widetilde{\mathcal{L}}_{\eta^C}\widehat{\sigma},\widehat{\vartheta}))-\rho_\pounds(\xi^V)(\text{tr}\mathpzc{B}(\widehat{\sigma},\widetilde{\mathcal{L}}_{\eta^C}\widehat{\vartheta})).
\end{align*}
Scince $\llb \eta^C,\xi^V\rrb=[\eta,\xi]^V$, then
\begin{align*}
\widetilde{\mathcal{L}}_{\eta^C}(\nabla^v\text{tr}\mathpzc{B})(\widehat{\xi},\widehat{\sigma},\widehat{\vartheta})&=\rho_\pounds(\xi^V)\lbrace \rho_\pounds(\eta^C)\text{tr}\mathpzc{B}(\widehat{\sigma},\widehat{\vartheta})\\
&-\text{tr}\mathpzc{B}(\widetilde{\mathcal{L}}_{\eta^C}\widehat{\sigma},\widehat{\vartheta})-\text{tr}\mathpzc{B}(\widehat{\sigma},\widetilde{\mathcal{L}}_{\eta^C}\widehat{\vartheta})\rbrace\\
&=\rho_\pounds(\xi^V)\lbrace \widetilde{\mathcal{L}}_{\eta^C}\text{tr}\mathpzc{B}(\widehat{\sigma},\widehat{\vartheta})\rbrace\\
&=\rho_\pounds(\xi^V)\lbrace(\text{tr}\widetilde{\mathcal{L}}_{\eta^C}\mathpzc{B})(\widehat{\sigma},\widehat{\vartheta})\rbrace\\
&=0.
\end{align*}
\end{proof}

\end{document}